\documentclass[12pt,oneside,english]{amsart}
\usepackage[T1]{fontenc}
\usepackage[latin9]{luainputenc}
\usepackage{geometry}
\usepackage{amstext}
\usepackage{amsthm}
\usepackage{amssymb}
\usepackage{graphicx}
\usepackage{mathtools}

\makeatletter
\numberwithin{equation}{section}
\numberwithin{figure}{section}
\theoremstyle{plain}
\newtheorem{thm}{\protect\theoremname}[section]
\theoremstyle{definition}
\newtheorem{defn}[thm]{\protect\definitionname}
\theoremstyle{definition}
\newtheorem{problem}[thm]{\protect\problemname}
\theoremstyle{plain}
\newtheorem{lem}[thm]{\protect\lemmaname}
\theoremstyle{remark}
\newtheorem{claim}[thm]{\protect\claimname}
\theoremstyle{plain}

\theoremstyle{remark}
\newtheorem*{acknowledgement*}{\protect\acknowledgementname}

\newcommand{\lmfrac}[2]{\mbox{\small$\displaystyle\frac{#1}{#2}$}}   

\newcommand{\tmfrac}[2]{\mbox{\large$\frac{#1}{#2}$}}

\usepackage{amsmath,amsthm,amsfonts,amssymb,amscd,flafter,pinlabel}
\usepackage[parfill]{parskip}
\usepackage[mathscr]{eucal}
\usepackage{graphics}
\usepackage[all,knot,arc]{xy}

\usepackage[usenames,dvipsnames]{color}
\usepackage{subfigure}
\usepackage{graphicx}
\usepackage{caption}
\usepackage{multirow}

\usepackage[colorlinks=true,linkcolor=blue,citecolor=blue,urlcolor=blue]{hyperref}

\def\mylist#1 {\ifx!#1\else\makebox[4em][r]{#1} \expandafter\mylist\fi}

\def\R{\mathbb{R}}
\def\C{\mathbb{C}}
\makeatother

\usepackage{babel}
\providecommand{\acknowledgementname}{Acknowledgement}
\providecommand{\claimname}{Claim}
\providecommand{\definitionname}{Definition}
\providecommand{\lemmaname}{Lemma}
\providecommand{\problemname}{Problem}
\providecommand{\propositionname}{Proposition}
\providecommand{\theoremname}{Theorem}

\begin{document}
\title{When does the Table Theorem imply a solution to the Square Peg Problem?}
\author{Stefan Friedl and Kenan Ince}
\begin{abstract}
We will explain the relationship between one of the most beautiful theorems in topology, namely Fenn's Table Theorem, and one of the most famous open problems in topology, namely the Square Peg Problem. 
\end{abstract}

\maketitle

\section{Introduction}

To introduce the Square Peg Problem, we need the following definition:
\begin{defn}
The image of an injective continuous map $S^{1}\to\mathbb{R}^{2}$ is called
a \emph{Jordan curve}. Given a Jordan curve $C$, we say that a rectangle
$R$ is \emph{inscribed in $C$ }if all vertices of $R$ lie on $C$.
\end{defn}

The following problem was first stated by Otto Toeplitz \cite{To1911} in 1911. 

\begin{problem}\textbf{\label{square-peg}(Square Peg Problem)}
 Does every Jordan curve admit an inscribed square of side length $>0$?
\end{problem}

This problem, and interesting variations thereof, have fascinated mathematicians ever since it was formulated. Many partial results have been obtained.
For example it is known that it is possible to inscribe a rectangle in every Jordan
curve \cite{Me1982}, and Lev Schnirelmann \cite{Sc1944} showed it is always possible to inscribe
a square if the Jordan curve is $C^{2}$, but the general case of
the problem is open. A discussion of the question and its history
is given in  \cite{Ma2014} and \cite{GL2021}.

The following theorem was first formulated and proven by Roger Fenn \cite{Fe1970}
in $1970$. (An alternative write up of the proof is given in \cite[Chapter~133.1]{Fr2023}.)

\begin{thm}\textbf{\emph{(Table Theorem)}}
\label{thm:Table-Theorem} Let $D\subset \R^2$ be a compact convex non-empty subset.
 Let $f\colon \mathbb{R}^{2}\to\mathbb{R}$ be
a continuous map so that $f(x)\geq0$ for all $x\in D$ and such that $f(x)=0$
for all $x\notin D$. Let $s>0$ be a real number. Then there exist
four points $a_{1},a_{2},a_{3},a_{4}\in\mathbb{R}^2$ with the following
properties:
\begin{enumerate}
\item the points form a square of side length $s$,
\item the center of the square lies in $D$,
\item we have $f(a_{1})=f(a_{2})=f(a_{3})=f(a_{4})$.
\end{enumerate}
\end{thm}

\begin{figure}[h]
\begin{center}
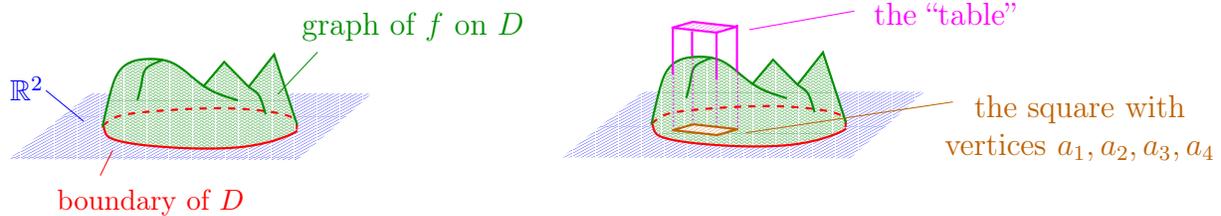 
\caption{Illustration of the Table Theorem.}
\end{center}
\end{figure}

In other words, the Table Theorem says that any given square table
can be placed on the ``ground'' defined by the graph of $f$ such
that all four legs of the table lie on the ground, the table is horizontal,
and the center of the tabletop lies in $D$.

It has been known for a long time that the Table Theorem~\ref{thm:Table-Theorem}
has implications for the Square Peg Problem. It was stated somewhat optimistically in \cite{Ma2014,Ta2017} that the Table Theorem~\ref{thm:Table-Theorem} implies the 
Square Peg Problem for convex Jordan curves. But to us it seems like that story is more complicated since, as we will see in Section~\ref{section:tt-trivial-solutions},  in many cases the Table Theorem has trivial solutions which do not tell us anything interesting. This leads us to the following definition.

\begin{defn}
\label{def:obtuse}We say that a subset $D$ of $\mathbb{R}^{2}=\mathbb{C}$
is \emph{obtuse }if, for any $x\in\partial D$, there exists a $v\in\mathbb{R}^{2}\setminus\{(0,0)\}=\mathbb{C}\setminus\{0\}$
and an angle $\theta>\frac{\pi}{2}$ such that
\[
T_{v,\theta}(x):=\{x+r\cdot  e^{i\phi}\cdot v:\phi\in[0,\theta]\text{ and }r\in[0,1]\}\subset D.
\]
\end{defn}


\begin{figure}[h]
\begin{center}
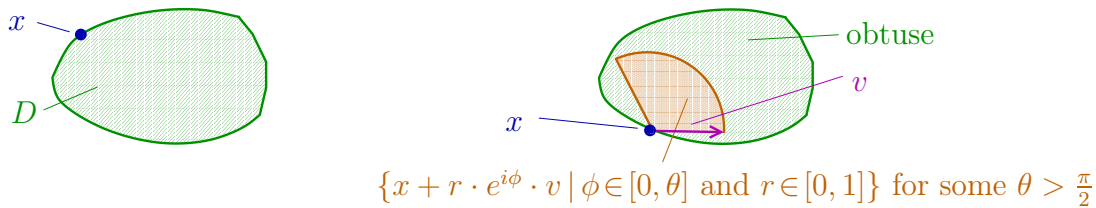 
\caption{\label{fig:An-obtuse-set}An obtuse set $D$ together with a picture
of $T_{v,\theta}(x)$.}
\end{center}
\end{figure}

In this short note we will show 
in the Main Lemma~\ref{lem:main-lemma} that the Table Theorem~\ref{thm:Table-Theorem} has a non-trivial solution if and only if $D$ is obtuse. 
We will use the Table Theorem~\ref{thm:Table-Theorem} to give a new proof for the following theorem.

\begin{thm}\textbf{\emph{(Main Theorem)}}
\label{thm:square-peg-for-Jordan-curves}Let $J$ be a Jordan curve
which is the boundary of a compact convex subset $D\subset\mathbb{R}^{2}$.
If $D$ is obtuse, then $J$ admits an inscribed square.
\end{thm}

Note though that the Square Peg Problem \ref{square-peg} for convex curves has been known for a long time;
in fact, it was proved by K.\ Zindler \cite{Zi1921} and C.\ M.\ Christensen \cite{Chr1950}.


\section{Preliminaries}\label{section:tt-trivial-solutions}
We return to the context of the Table Theorem.
Note that given a compact convex non-empty subset  $D\subset \R^2$,  for $s$ very large, we can simply place the legs of the
table outside of $D$ (in fact, by continuity of $f$, outside the
interior of $D$) while keeping the center of the table in $D$, since
then $f(a_{1})=f(a_{2})=f(a_{3})=f(a_{4})=0$. Motivated by such a
``trivial solution'' to the Table Theorem, we make the following
definition:

\begin{defn}
Let $D$ be a subset of $\R^2$ and let $s\in \R$.  We say $D$ is \emph{$s$-trivial }if
there exists a square in $\mathbb{R}^{2}$ with side length $t\in(0,s]$ such
that the vertices of the square lie outside the interior of $D$ while
the center of the square lies in $D$. Otherwise, we say that $D$
is \emph{$s$-nontrivial.} In other words, $D$ is $s$-nontrivial if, for all squares with side length $t\in(0,s]$ with center in $D$, at least one vertex of the square lies in the interior of $D$.
\end{defn}

Note that, by our definition, if $D$ is $s$-trivial for some $s$, then certainly $D$ is $s'$-trivial for all $s'>s$. Moreover, every non-empty compact set $D$ is $s$-trivial for all sufficiently large $s$; for example, one may take $s\geq\operatorname{diam}(D)$ and center the trivial square of side length $s$ at any point in $D$.

\begin{figure}[h]
\begin{center}
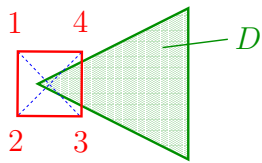 
\caption{Example of a subset $D\subset \C$ that is $s$-trival for all $s>0$.}
\end{center}
\end{figure}

The following lemma is the main technical result we prove in this note.

\begin{lem}\textbf{\emph{(Main Lemma)}}\label{lem:main-lemma}
 Let $D\subset\mathbb{R}^{2}$ be a compact, convex, non-empty set. Then $D$
is obtuse if and only if $D$ is $s$-nontrivial for some $s>0$.
\end{lem}

Thus, Lemma \ref{lem:main-lemma} is precisely the condition needed for the Table Theorem to imply a \emph{nontrivial} solution to the Square Peg Problem, in the sense that we do not simply make the square huge enough that all of its vertices lie outside of $D$.

The next section will be devoted to a proof of the Main Lemma~\ref{lem:main-lemma}. In Section
\ref{sec:nontriv-implies-square-peg}, we will show that the Main Lemma~\ref{lem:main-lemma} implies Theorem
\ref{thm:square-peg-for-Jordan-curves}.
\section{Proof of the Main Lemma~\ref{lem:main-lemma}}

\begin{defn}
Let $D\subset \R^2$ be compact and convex.
We consider the function
\[ \begin{array}{rcl} f_D\colon D&\to& \R_{\geq 0}\\
x&\mapsto  &\sup\{\|v\|:\mbox{ $v\in \R^2$ such that  there exists a $\gamma>\frac{\pi}{2}$ with $T_{v,\gamma}(x)\subset D$}\}.\end{array}\]
Note that we can always take $v=0$, i.e.\ we take the supremum over a non-empty set.
\end{defn}

\begin{lem}\label{claim:f-is-lower-semicontinuous}
Let $D\subset \R^2$ be compact and convex. The function $f_D\colon D\to \R$ defined
above is lower semicontinuous. That is, for any $x\in D$
and any $\epsilon>0$, there exists a $\delta>0$ such that, for all $y\in D$
with $\|x-y\|<\delta$, $f_D(y)>f_D(x)-\epsilon$.
\end{lem}

\begin{proof}
 Let $x\in D$ and $\epsilon>0$ be arbitrary. If $f_D(x)=0$, then there is nothing to show. Thus we can now assume that $f_D(x)>0$. By the definition
of $f_D$ as a supremum there exists
a non-zero vector $v\in\mathbb{R}^{2}$ such that $\|v\|>f_D(x)-\frac{\epsilon}{2}$
and  an angle $\gamma\in ( \frac{\pi}{2},\pi)$  with $T_{v,\gamma}(x)\subset D$.

It suffices to show that there exists $\delta>0$ such that, for all
$y\in D$ with $\|x-y\|<\delta$, we can find a $w\in\mathbb{R}^{2}$
with $\|w\|>f_D(x)-\epsilon$ and an angle $\beta\in (\frac{\pi}{2},\pi)$
with $T_{w,\beta}(y)\subset D$.

We set $A:=x+v$ and $B:=x+e^{i\gamma}\cdot v$. 
We consider the map 
\[ \begin{array}{rcl}g\colon D\setminus \{A,B\}&\to & [0,\pi]\\
y&\mapsto & \sphericalangle_{AyB}:=\arccos\bigg(\lmfrac{\langle A-y,B-y\rangle}{\|A-y\|\cdot \|B-y\|}\bigg).\end{array}\]
It is clear that $g$ is continuous.
Since $g(x)=\gamma>\frac{\pi}{2}$ we see that  there exists a $\mu>0$ such that $g(y)>\frac{\pi}{2}$ for all $y\in D$ with $\|x-y\|<\mu$. We set $\delta:=\operatorname{min}\{\mu,\frac{\epsilon}{2}\}$. 
We claim that $\delta$ has the desired property.

Let $y\in D$ with $\|x-y\|<\delta$. We claim that 
$w:=A-y$ and $\beta:=g(y)$ have the desired properties.
First note that it follows from  $\delta\leq \mu$
that  $\sphericalangle_{AyB}=g(y)>\frac{\pi}{2}$.  

By convexity of $D$ we know that for each $\varphi\in [0,\gamma]$
the segment from $y$ to  $z=x+e^{i\varphi}\cdot v$ is contained in $D$.
Furthermore note that  since $\delta<\frac{\epsilon}{2}$
we see that for each $z=x+e^{i\varphi}\cdot v$ we have $\|z-y\|\geq \|z-x\|-\|y-x\|>f_D(x)-\frac{\epsilon}{2}-\frac{\epsilon}{2}=f_D(x)-\epsilon$. 
It follows from this discussion that $T_{w,\beta}(y)\subset D$. 
\end{proof}

\begin{figure}[h]
\begin{center}
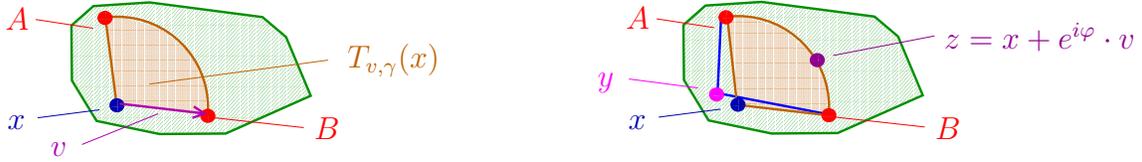 
\caption{Illustration for the proof of Lemma~\ref{claim:f-is-lower-semicontinuous}.}
\end{center}
\end{figure}

We now turn to the proof of  the Main Lemma~\ref{lem:main-lemma}.
Thus let $D\subset\mathbb{R}^{2}$ be a compact, convex, non-empty set.
We want to show the following:
\[ \mbox{$D$ is obtuse}\quad \Longleftrightarrow\quad 
\begin{array}{c}\mbox{$D$ is $s$-nontrivial for some $s>0$.}\end{array}\]
We prove the two directions separately.

\begin{proof}[Proof of the Main Lemma~\ref{lem:main-lemma} $\Longrightarrow$]
 Let $D\subset\mathbb{R}^{2}$ be a compact, convex, obtuse, non-empty set. We need to show 
that there exists a real number $s>0$ such that
$D$ is $s$-nontrivial.
 We continue with the above notation; 
in particular, we consider the function $f_D\colon D\to \R$.
Note that for any $x$ in the interior of $D$ we evidently have $f_D(x)>0$. Furthermore note that it follows from the fact that $D$ is obtuse that for any $x\in \partial D$ we also have $f_D(x)>0$. In summary we have shown that $f_D\colon D\to \R_{\geq 0}$ is non-zero everywhere.
Recall  that every lower semicontinuous function on a non-empty compact subset achieves a global minimum (see e.g.\ \cite[Theorem~II.10.1]{Cho1966}).
Thus it follows from  Lemma~\ref{claim:f-is-lower-semicontinuous} that  $f_D\colon D\to \R$ achieves a global
minimum. Denote
\[
s:= \min\{f_D(x):x\in D\}.
\]
As we mentioned above, $f_D$ is non-zero everywhere, thus we see that $s>0$. 

We claim that $D$ is $s$-nontrivial. 
Let $x\in D$. We need to show that every square centered at $x$ of side length $\leq s$ has least one vertex lying in $\operatorname{int}(D)$. 

By definition of $f_D$ and by definition of a supremum, we know that there exists
a $v\in \R^2$ with length $v>\frac{1}{\sqrt{2}}\cdot f_D(x)$ and a $\beta>\frac{\pi}{2}$ such that 
$T_{(v,\beta)}(x)\subset D$. Since $\beta>\frac{\pi}{2}$, it follows from elementary geometry that any square centered at $x$ with side length $\leq \sqrt{2}\cdot \|v\|$ contains a vertex in $\operatorname{int}(T_{(v,\beta)}(x))$. Since $\sqrt{2}\cdot \|v\|\geq  f_D(x)\geq s$, we see that any square centered at $x$ of side length $\leq s$ admits at least one vertex that lies in $\operatorname{int}(D)$.

\end{proof}

\begin{proof}[Proof of the Main Lemma~\ref{lem:main-lemma} $\Longleftarrow$]
Let $D\subset\mathbb{R}^{2}$ be a compact, convex set. 
We need to show that  if  $D$
is non-obtuse, then the subset $D$ is $s$-trivial for all $s>0$.

Thus we assume that $D$ is non-obtuse.  Recall that this means that there exists an $x\in \partial D$ such that for all 
$v\ne 0$ and all $\theta>\frac{\pi}{2}$ we have $T_{v,\theta}(x)\not\subset D$. 
After a translation we can assume, without loss of generality, that $x=0$.

Let $s>0$. We need to show that $D$ is $s$-trivial. It suffices to show that there exists a square in $\R^2$ with the following properties:
\begin{enumerate}
\item The center of the square is $x=0$.
\item The distance from $x$ to the four vertices equals $d:=\frac{1}{\sqrt{2}}s$. (Which implies that the side length is $s$.)
\item All four vertices lie outside of the interior of $D$. 
\end{enumerate}
We set 
\[ B:=\{ v\in S^1\,|\, \R_{>0}\cdot v\cap D\ne\emptyset\}.\] 
\begin{figure}[h]
\begin{center}
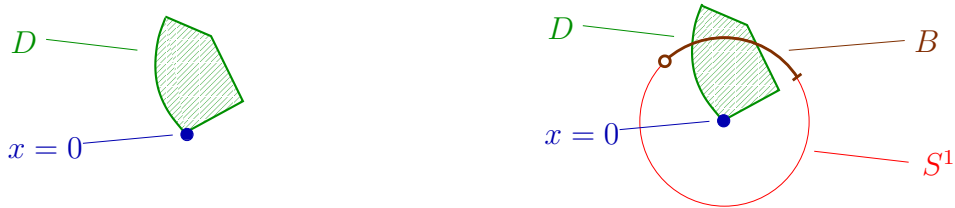 
\caption{\label{fig:non-obtuse-B} A depiction of a non-obtuse set $D$ together with the corresponding set $B$.}
\end{center}
\end{figure}

If $B$ consists of just two opposite points $v,-v$, then the desired square is given by the 
vertices $v\cdot e^{i(\pi/4+k\cdot \pi/2)}$, $k=0,1,2,3$. Thus in the following we will assume that $B$ does not just consist of two opposite points. 

\begin{claim}
The subset $B\subset S^1$ is path-connected.
\end{claim}

Let $v\ne w\in B$. We need to show that $v$ and $w$ are connected by a path in $B$.
First we consider the case that  $v\ne -w$. 
We pick $r>0,s>0$ such that $r\cdot v,s\cdot w\in D$. We set $\tilde{v}:=r\cdot v$ and $\tilde{w}:=s\cdot w$. 
We consider the triangle formed by $0,\tilde{v},\tilde{w}$. 
By convexity of $D$ this triangle is contained in $D$. This implies that 
all points of the form \[
\frac{ \tilde{v}\cdot (1-t)+\tilde{w}\cdot t}{\|\tilde{v}\cdot (1-t)+\tilde{w}\cdot t\|}
\] are contained in $B$. (The fact that $v\ne -w$ implies that 
$\tilde{v}\cdot (1-t)+\tilde{w}\cdot t\ne 0$ for all $t\in [0,1]$.)
Thus $v$ and $w$ are connected by a path in $B$. 

Now assume that $v=-w$. It follows from the discussion preceding the claim that there exists a $u\in B$ with $u\ne \{v,-w\}$. But by the above we can connect $v$ to $u$ by a path in $B$ and we can connect $u$ to $w$ by a path in $B$. Thus we can connect $v$ to $w$ by a path in $B$.
This concludes the proof of the claim.

Since $B\subset S^1$ is connected, there exists an interval $I\subset \R$ of length $\leq 2\pi$ such that 
\[ B=\{ e^{it}\,|\, t\in I\}.\]
We set $\varphi:=\operatorname{inf}(I)$ and  $\psi:=\operatorname{sup}(I)$.

\begin{claim}
We have $\psi-\varphi\leq \frac{\pi}{2}$.
\end{claim}

Suppose that $\psi-\varphi>\frac{\pi}{2}$. We pick $\varphi',\psi'$ with 
$\varphi<\varphi'<\psi'<\psi$ and such that $\psi'-\varphi'\in (\frac{\pi}{2},\pi)$. 
By definition of supremum and infimum, and since $I$ is an interval,  we know that $\varphi',\psi'\in I$, which implies that $e^{i\varphi'},e^{i\psi'}\in B$, which in turn implies that there exist $r>0$ and $s>0$
such that $r\cdot e^{i\varphi'}\in D$ and $s\cdot e^{i\psi'}\in D$. 
By convexity of $D$, the triangle given by $0,r\cdot e^{i\varphi'},s\cdot e^{i\psi'}$ is contained in $D$. An elementary geometric argument now shows that there exists a $\lambda>0$ such that
$T_{\lambda e^{i\varphi '},\psi'-\varphi'}(0)\subset D$. Since $\psi'-\varphi'>0$, this contradicts our choice of $x$.
This concludes the proof of the claim.

Now we consider the square that is given by the vertices $d\cdot e^{i(\varphi+k\cdot \frac{\pi}{2})}$, $k=0,1,2,3$. It is clear that (1) and (2) are satisfied. Thus it remains to prove the following claim.

\begin{claim}
All four vertices lie outside of the interior of $D$.
\end{claim}

We start out with an observation: if $r\cdot e^{i\beta}$ lies in the interior of $D$ then, since the interior of $D$ is open, there exists an $\epsilon>0$ such that 
$r\cdot e^{i(\beta+\sigma)}\in D$ for all $\sigma\in (-\epsilon,\epsilon)$, which implies that $(\beta-\epsilon,\beta+\epsilon)\subset I$. 
The claim follows from this observation and the definition of infimum and supremum.
\end{proof}

\section{\label{sec:nontriv-implies-square-peg}Nontriviality implies a solution to the Square Peg Problem}
We can now give a proof of the Main Theorem~\ref{thm:square-peg-for-Jordan-curves}. The key idea behind the  argument we provide is in principle well-known, see e.g.\ \cite{Ma2014}. But note that the role of the obtuseness has not been elucidated before. 

\begin{proof}
Let $D\subset \R^2$ be a compact convex obtuse non-empty subset.
We need to show that there exist 
four points $a_1,a_2,a_3,a_4\in \partial D$ which form a square of side length $>0$.
After a translation we can assume that the origin $0$ is contained in the interior of $D$.
Given $x\in D\setminus \{0\}$ we set
\[ \rho(x)\,\,:=\,\, \sup\big\{ \|r\cdot x\|\,\big|\, r\in \R_{> 0}\mbox{ and } r\cdot x\in D\big\}\,\,\in\,\R_{>0}.\]
It follows from an elementary argument, see e.g.\ \cite[Chapter~11.3]{Be2009}, that $\rho$ is continuous. 
We consider the  map 
\[ \begin{array}{rcl} f\colon \R^2&\to & [0,1]\\
x&\mapsto & \left\{ \begin{array}{ll} 1-\|x\|\cdot \tmfrac{1}{\rho(x)}, &\mbox{ if }x\in D\setminus\{0\},\\ 
1,&\mbox{ if }x=0,\\
0,&\mbox{ if }x\not\in D.\end{array}\right.\end{array}\]
Since $\rho$ is continuous one can easily verify that $f$ is continuous.
 Since $D$ is obtuse we obtain from the Main Lemma~\ref{lem:main-lemma} that there exists a $d>0$ such that $D$ is $d$-non trivial.
By the Table Theorem~\ref{thm:Table-Theorem}  there exist four points $b_1,b_2,b_3,b_4\in \R^2$ with the following properties:
\begin{enumerate}
\item the points form a square of side length $d$,
\item the center of the square lies in $D$,
\item we have $f(b_1)=f(b_2)=f(b_3)=f(b_4)$.
\end{enumerate}
Since $D$ is $d$-non trivial, we see that at least one vertex lies in the interior of $D$. But since $f$ is non-zero on the interior of $D$, we see that the common $f$-value, let's call it $y$, lies in the open interval $(0,1)$. It follows immediately from the definition of $f$ that the four points $b_1,b_2,b_3,b_4$ lie on the subset
$(1-y)\cdot \partial D$. In other words, if we multiply these four points by $\frac{1}{1-y}$, then we obtain the desired four points $a_1,a_2,a_3,a_4$ on $\partial D$. 
\end{proof}

\begin{figure}[h]
\begin{center}
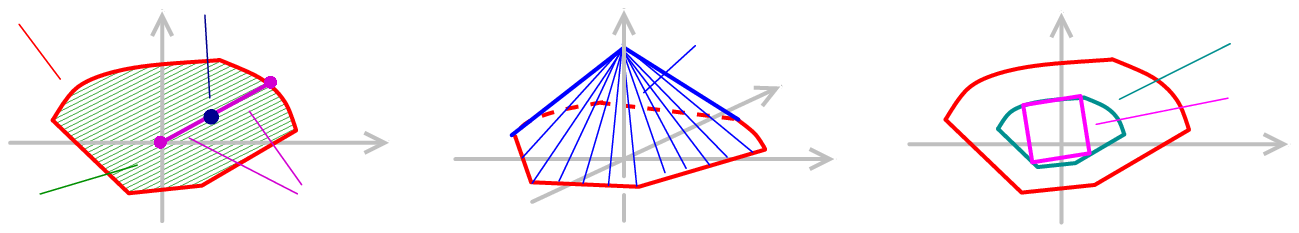
\end{center}
\end{figure}

\begin{acknowledgement*}
The first author was supported by the CRC 1085 ``Higher Invariants'' at the University of Regensburg.
The second author would like to acknowledge the Westminster College Gore Individual Summer Grant and the University of Regensburg for funding to support this collaboration.
\end{acknowledgement*}

\end{document}